\documentclass[a4paper,dvips,12pt,oneside]{article}
\usepackage[makeroom]{cancel}
\usepackage{graphicx}
\usepackage{epstopdf}
\usepackage{subfig}
\usepackage{color}
\usepackage{float}
\usepackage{leftidx}
\usepackage{bigints}
\usepackage[affil-it]{authblk}
\usepackage{amssymb,amsmath,dsfont,url,epsfig}
\usepackage[left=1in, right=1in, top=1.1in, bottom=1.1in, includefoot, headheight=13.6pt]{geometry}
\usepackage[T1]{fontenc}
\usepackage{titlesec, blindtext, color}
\definecolor{gray75}{gray}{0.75}

\newcommand{\sln}{\linespread{1}}
\newcommand*{\email}[1]{\href{mailto:#1}{\nolinkurl{#1}} } 
\titleformat{\chapter}[block]{\LARGE\bfseries\sln}{Chapter \thechapter}{11pt}{\newline\huge\bfseries}
\newtheorem{theorem}{Theorem}[section]
\newtheorem{remark}{Remark}[section]
\newtheorem{definition}{Definition}[section]

\newenvironment{proof}{\paragraph{Proof:}}{\hfill$\square$}
\newtheorem{lemma}{Lemma}[section]
\newtheorem{proposition}{Proposition}[section]

\begin{document}
\title{The Isoperimetric Problem in Randers Poincar\'e Disc}

\author[1]{Arti Sahu Gangopadhyay
  \thanks{E-mail: \texttt{arti.sahu@bhu.ac.in}}}
\affil[1,4]{DST-CIMS, Banaras Hindu University, Varanasi-221005, India}

\author[2]{Ranadip Gangopadhyay
  \thanks{E-mail: \texttt{gangulyranadip@gmail.com}}}

\author[3]{Hemangi Madhusudan Shah
  \thanks{E-mail: \texttt{hemangimshah@hri.res.in}}}
\affil[3]{Harish-Chandra Research Institute, A CI of Homi Bhaba National Institute, Chhatnag Road, Jhunsi, Prayagraj-211019, India}

\author[4]{Bankteshwar Tiwari
  \thanks{E-mail: \texttt{btiwari@bhu.ac.in}}}

\maketitle

\begin{abstract}
It is known that a simply connected Riemann surface satisfies the isoperimetric equality if and only if it has constant  Gaussian curvature. In this article, we show that Randers Poincar\'e disc satisfies the isoperimetric equality with respect to different volume forms. These metrics do not necessarily have constant (negative) flag curvature. Thus we show that the Osserman's result \cite{RO} in the Riemannian case can not be extended to the corresponding Finslerian case.\\


keywords: Isoperimetric problem, Randers Poincar\'e disc , calculus of variations, conjugate points. \\
\end{abstract}

\section{Introduction}
The isoperimetric problem is a more than two millenia old problem which says that: among all the simple closed curve with fixed perimeter, the circle encloses the maximum area.   For the non-zero constant curvature spaces the isoperimetric inequality is given by
\begin{equation*}
L^2\ge 4\pi A-kA^2,
\end{equation*}
where $L$ is the length of a simple closed curve, $A$ is the area bounded by the curve and  $k$ is the curvature of the space. The isoperimetric problem has several important applications in mathematical physics. In 1903,  Poincar\'e stated with incomplete (variational) proof that of all the solids with a given volume the sphere has the minimum electrostatic capacity  and  Szeg\"{o} gave the first complete proof for Poincar\'e's theorem on electrostatic capacity  in 1930 \cite{GP}. In recent years scientists came to know that an incredibly intense magnetic fields can arise in several realistic physical system (for example, certain kind of neutron star). To study these kind of  physical problems, the knowledge of the isoperimetric problem in hyperbolic manifolds and asymptotic hyperbolic manifolds turns out to be essential \cite{BM}. The study of the isoperimetric inequality has been done extensively by several mathematicians, for example see \cite{AGBE,MG,MP,STY}. In Finsler geometry,  Zhou studied the isoperimetric problem for the Berwald metric, \begin{equation*}
F(x,y)=\frac{(\sqrt{|y|^2-(|x|^2|y|^2)-\langle x,y\rangle^2}+\langle x,y\rangle)^2}{(1-|x|^2)(\sqrt{|y|^2-(|x|^2|y|^2)-\langle x,y\rangle^2)}}
\end{equation*}  with respect to the Busemann-Hausdorff volume form \cite{LZ1}. Zhan and Zhou studied this problem with respect to the Holmes-Thompson volume form \cite{LZ2} and showed that the circles $\gamma_0(t)=(a\cos t, a\sin t), a<1$, are the local solutions of the isoperimetric problem. Li and Mo \cite{LM} studied the isoperimetric problem for the two dimensional Funk metric
\begin{equation*}
F(x,y)=\frac{(\sqrt{|y|^2-(|x|^2|y|^2)-\langle x,y\rangle^2}+\langle x,y\rangle)}{2(1-|x|^2)}.
\end{equation*}
Recently, Sahu et al., \cite{AS} studied the isoperimetric problem for a special class of Randers metrics, which are deformations of Euclidean metrics by some suitable one-form $\beta$.\\

In this paper, we study the isoperimetric problem in Randers Poincar\'e disc model. In this model we consider the unit disc $\mathbb{D}$  in $\mathbb{R}^2$ with the Poincar\'e metric $\displaystyle{\alpha=\frac{2\sqrt{(dx^1)^2+(dx^2)^2}}{1-(x^1)^2-(x^2)^2}}$. Consider the Randers metric $F=\alpha+\beta$, which can be thought deformation of the Poincar\'e metric $\alpha$ by an exact one-form $\beta=df$, where $f:\mathbb{D}\to \mathbb{R}$ is a  smooth function with $\|\beta\|_{\alpha}$ is constant and $\|\beta\|_{\alpha}<1$. We call the so obtained Finsler space $(\mathbb{D},F)$ as \textit{Randers Poincar\'e disc}. In the present paper we study the isoperimetric problem in the Randers Poincar\'e disc $(\mathbb{D}, F)$.\\
In general, a metric of the form $F=\alpha+\beta$, where $\alpha$ is any Riemannian metric and $\beta$ is a one form with $\|\beta\|_{\alpha}<1$, is called the Randers metric and has several application in physics and biology.  Bao et. al., \cite{BRS} proved that the solution of a Zermelo Navigation problem
in a Riemannian manifold $(M, h)$ with a time independent vector field $W$ satisfying $h(W, W) < 1$, gives a unique Randers metric on the manifold $M$ and conversely.\\

Now we state our main results of this paper

\begin{proposition} \label{prop2.001}
 Let $(\mathbb{D},\alpha)$ be the hyperbolic Poincar\'e disc where, $\displaystyle{\alpha= \frac{2\sqrt{(dx^1)^2+(dx^2)^2}}{1-(x^1)^2-(x^2)^2}}$ and $\beta=df$ be an exact one form with constant norm and $||\beta||_{\alpha}<1$. Then $\beta$ can be explicitely expressed by
 \begin{equation*}
    \beta= b\frac{x^1 dx^1 +x^2 dx^2}{(1-(x^1)^2-(x^2)^2)\sqrt{(x^1)^2+(x^2)^2}} 
 \end{equation*} with $0<b<1$. Hence $F=\alpha+\beta$ is the Randers Poincar\'e metric on $\mathbb{D}$.
\end{proposition}
\begin{proposition} \label{prop2.0011}
 Let $(\mathbb{D}, F=\alpha+\beta)$ be the Randers Poincar\'e disc as described in the above Proposition. Then the space $(\mathbb{D}, F)$ does not have constant negative flag curvature.
\end{proposition}

\begin{theorem}\label{thm2.1}
Let $(\mathbb{D}, F=\alpha+\beta)$ be the Randers Poincar\'e disc. Then the circles of the form $\gamma_0=(a\cos t,a\sin t)$, are solutions of the isoperimetric problem, with respect to the volume forms: Busemann-Hausdorff, Holmes-Thompson, Maximum and Minimum.
\end{theorem}
 Thus our result also holds for spaces which may not have constant negative flag curvature. Note that in \textit{contrast} to the Riemannian case, where it is known that the isoperimetric equality holds on a Riemann surface if and only if it is of constant sectional curvature. For details see Osserman's result, that is, corollary of Theorem $6$, page-$9$ of \cite{RO}.
\section{Preliminaries}
Let $ M $ be an $n$-dimensional smooth manifold. $T_{x}M$ denotes the tangent space of $M$
 at $x$. The tangent bundle of $ M $ is the disjoint union of tangent spaces, $ TM:= \sqcup _{x \in M}T_xM $. We denote the elements of $TM$ by $(x,y)$, where $y\in T_{x}M $ and $TM_0:=TM \setminus\left\lbrace 0\right\rbrace $, the slit tangent bundle of $M$. \\
 \begin{definition}
\textnormal{ \cite{SSZ} A Finsler metric on $M$ is a function $F:TM \to [0,\infty)$ satisfying the following conditions:
 \begin{itemize}
 \item[(i)] $F$ is smooth on $TM_{0}$,
  \item[(ii)] $F$ is positively 1-homogeneous on the fibers of the tangent bundle $TM$,
   \item[(iii)] The Hessian of $\displaystyle{\frac{F^2}{2}}$ with element $\displaystyle{g_{ij}=\frac{1}{2}\frac{\partial ^2F^2}{\partial y^i \partial y^j}}$ is positive definite on $TM_0$, where $y=y^i\frac{\partial}{\partial x^i}$.
 \end{itemize}
 The pair $(M,F)$ is called a Finsler space and $g_{ij}$ is called the fundamental tensor.}
 \end{definition}
 \begin{definition}\textnormal{ \cite{SSZ}
Let $\gamma:[a,b]\to M$ be a piecewise smooth curve on the Finsler manifold $(M,F)$. Then the {\it length} of the curve $\gamma$ is given by 
 \begin{equation}\label{eqn2.1}
     \int_a^b F(\gamma(t),\dot{\gamma}(t))dt.
 \end{equation}}
 \end{definition}
  \vspace{0.1in}
  
 It is known that, there is a canonical volume form in a Riemannian manifold $(M^n,\alpha),~\alpha=\sqrt{a_{ij}dx^idx^j}$, given by $dV=\sqrt{\det(a_{ij})}dx$, whereas in the Finsler manifold there are several volume forms, some of them are defined below:
  \begin{definition}
\textnormal{\cite{SSZ, BYW} For a Finsler manifold $(M^n,F)$, the {\it Busemann-Hausdorff} volume form is defined as $dV_{BH}=\sigma_{BH}(x)dx$, where
     \begin{equation}\label{2.1.76}
     \sigma_{BH}(x)=\frac{vol(B^n(1))}{vol\left\lbrace (y^i)\in T_xM : F(x,y)< 1 \right\rbrace }.
     \end{equation}
   The {\it Holmes-Thompson} volume form is defined as $dV_{HT}=\sigma_{HT}(x)~dx$, where 
     \begin{equation}\label{2.02}
     \sigma_{HT}(x)=\frac{1}{vol(B^n(1))}\int_{F(x,y)<1}\det(g_{ij}(x,y))dy.
     \end{equation}
     Here $B^n(1)$ is the Euclidean unit ball in $\mathbb{R}^n$ and $vol$ is the Euclidean volume.}\\
     \textnormal{The {\it maximum} volume form $dV_{max}$ and the {\it minimum} volume form $dV_{min}$ of a Finsler metric $F$ is defined as,
 \begin{equation}
 dV_{\max}=\sigma_{\max}(x)dx, \quad  dV_{\min}=\sigma_{\min}(x)dx,
 \end{equation}
 where, $\sigma_{\max}(x)=\max\limits_{y \in I_x}\sqrt{\det(g_{ij}(x,y))}$ \quad and \quad $\sigma_{\min}(x)=\min\limits_{y\in I_x}\sqrt{\det(g_{ij}(x,y))}$;  $I_x=\left\lbrace y\in T_xM| F(x,y)=1\right\rbrace$ is the Indicatrix of the Finsler metric $F$ at the point $x$.}
 \end{definition}  
\begin{lemma}\label{lem2.1}\textnormal{\cite{SSZ, BYW}}
    The Busemann-Hausdorff volume form of a Randers metric $F=\alpha+\beta$ is given by,
         \begin{equation}
         dV_{BH}=(1-\|\beta\|^2_{\alpha})^{\frac{n+1}{2}}dV_{\alpha}
         \end{equation}
          and the Holmes-Thompson volume form is given by,
         
         \begin{equation}
         dV_{HT}=dV_{\alpha}.
         \end{equation}
         The maximum and the minimum volume forms of a Randers metric $F=\alpha+\beta$ are, given by,
       \begin{equation}
       dV_{\max}=(1+\|\beta\|_{\alpha})^{n+1}dV_{\alpha}, \quad
       dV_{\min}=(1-\|\beta\|_{\alpha})^{n+1}dV_{\alpha}.
       \end{equation}
    \end{lemma}
 \section{The sufficient conditions for the solvability of the isoperimetric problem in Randers  Poincar\'e disc}
  In this section, we discuss about the sufficiency condition for the solvability of the isoperimetric problem essentially due to Hesteens \cite{MH} in a slightly modified form. Towards this, we shall focus on the calculus of variations in a parameteric form in $2$-dimensional Euclidean space.
 First we recall some definitions. More details can be found in \cite{OB, MH, LZ1,  LZ2}.\\
  
 \textbf{Admissible curve:} A continuous curve $\gamma$, defined on $\left[t_0,t_1 \right]$, in a manifold is called an {\it admissible curve}, if there exists a partition $ P := \left\lbrace t_0=a_{0}<a_{1}<...<a_{k}=t_1\right\rbrace$ of the interval $[t_0,t_1]$ such that $\gamma $ has continuous derivative in each subinterval $\left[ a_{i},a_{i+1} \right]$, for all $i = 0,1,...,k-1 $.\\
  
 \textbf{The Isoperimetric problem in Randers Poincare disc:}  
 We discuss the isoperimetric problem in Randers Poincare disc $\mathbb{D}$, which is an open unit disc in $\mathbb{R}^2$ equipped with Randers metric.  Let  the parametric form  of an arc in $\mathbb{D}$ is
$\gamma(t) := (x^1(t), x^2(t))$, where the functions $x^{i}(t)$ are assumed to be of class $C^2$ and 
  $({\dot{x}^{1}(t)})^2 + ({\dot{x}^{2}(t)})^2\neq 0$ for any $t\in [t_{0}, t_{1}]$.\\
  
  An integral
  $$ \int_{t_{0}}^{t_1}  g(t, \gamma(t), \dot{\gamma}(t)) \; dt,$$
 can be shown to be independent of the parameterization  of the 
 arc $\gamma$ if and only if the integrand $f$ is  positively homogeneous in $\dot{\gamma}$ of degree one. \\ \\
 Consider the Poincar\'e disc $({\mathbb{D}},F)$ with some volume form $dV_{\mu}$. The  isoperimetric 
 problem under consideration can be stated as follows:
 \begin{equation*} 
 \begin{split} 
 \text{To find a curve in $(\mathbb{D},F)$ which maximizes}~ \{A_{\mu}(D)\vert ~ D ~ \mbox{is the domain enclosed by an  }\\ \text{arbitrary simple closed piecewise smooth curve under the constrained that  }\\ \text{ the curve's length} ~ L=l\}.\hspace{7.5cm}
 \end{split} 
 \end{equation*} 
Here $A_{\mu}$ denotes the area with respect to the volume form $dV_{\mu}$ and $L$ denotes the length of the curve. \\
 
 Thus we consider  the set of all admissible curves joining two fixed points for which the integral $L = \int\limits_{t_0}^{t_1}g(t, \gamma(t),\dot{\gamma}(t))dt$ takes a constant value $l$. To solve the problem, we look for the admissible curve $\gamma(t)=(x^1(t),x^2(t))$ in $\mathbb{D}$ for which the integral $A=\int\limits_{t_0}^{t_1}f(t, \gamma(t),\dot{\gamma}(t))dt$ has its \textit{extremum}.\\
 
\textbf{Lagrange functional:} To solve the isoperimetric problem
in $\mathbb{D}$, we first consider the Lagrange functional:
 \begin{equation}\label{eq2.01}
 J(\gamma)=A+\lambda L=\int_{t_0}^{t_1}h(t,\gamma, \dot{\gamma};\lambda)dt,
 \end{equation}
where, 
\begin{equation}\label{2.03}
h=f+\lambda g,
\end{equation}
and $\lambda$ is a Lagrange multiplier. \\

\textbf{Isoperimetric extremal:} An admissible curve $\gamma(t)$ is said to be an {\it isoperimetric extremal}, if it satisfies the following Euler-Lagrange equations:
\begin{equation}\label{eq2.0.1}
\frac{\partial h}{\partial x^i}- \frac{d}{dt}\left(\frac{\partial h}{\partial \dot{x}^i} \right) =0, ~ i=1,2. 
\end{equation}

\vspace{0.1in}

\textbf{Normality of an admissible curve:} An admissible curve $\gamma(t)$ satisfying the Euler-Lagrange equations is said to be \textit{normal}, if $P_1$ and $P_2$ are never zero functions, where
\begin{equation}\label{2.1.64}
 P_i=g_{x^i}-\frac{d}{dt}g_{\dot{x}^i}, \quad i=1,2.   
\end{equation}

\textbf{Weierstrass function:} 
The {\it Weierstrass} function $E$, for the function $h$ defined by
(\ref{2.03}), is given by
\begin{equation}\label{2.25}
E(t, \gamma,\dot{\gamma}, u)=h(t, \gamma,u)-h(t,\gamma, \dot{\gamma})-\sum\limits_{j=1}^{2}(u^{j}-\dot{x}^{j})\frac{\partial h(t, \gamma, \dot{\gamma})}{\partial \dot{x}^j},
\end{equation}
where $h(t, \gamma, \dot{\gamma})=f(t, \gamma, \dot{\gamma})+\lambda g(t, \gamma, \dot{\gamma})$.\\

\textbf{Second variation of $\boldmath{J}\unboldmath$ along a curve:} The \textit{second variation} of $J$ along the curve $\gamma_0$ is of the form
\begin{equation}
J''(\gamma_0,y)=\int\limits_{t_0}^{t_1}2\omega(t,y(t),\dot{y}(t))dt,
\end{equation}
where,  $2\omega(t,y(t),\dot{y}(t)) =\sum\limits_{i,j=1}^{2} h_{x^ix^j}y^iy^j+2h_{x^i\dot{x}^j}y^i\dot{y}^j+h_{\dot{x}^i\dot{x}^j}\dot{y}^i\dot{y}^j$. \\ 
With this background a sufficiency theorem proved by 
Hestenes \cite{MH}, for a strong maximum of the isoperimetric problem, with slight modification can be stated as follows.
\begin{theorem}
[{\bf Sufficiency Theorem}
\textnormal{\cite{LZ1}}\label{thm2.2}]
Let $\gamma_0$ be an admissible curve. Suppose there exist $\lambda_0$ such that, relative to the function
$$ J(\gamma) = \int\limits_{t_0}^{t_1} h(t,\gamma, \dot{\gamma};\lambda)dt,$$
(i) $\gamma_0$ is isoperimetric extremal, that is, it satisfies the Euler-Lagrange equations \eqref{eq2.0.1};\\
(ii) $\gamma_0$ is normal, that is, \eqref{2.1.64} is satisfied;\\
(iii) the Weierstrass function $E(\gamma, \dot{\gamma} , u) < 0$ for every admissible set $(x^i, u^i)$ with $u \ne k\dot{\gamma} (k > 0)$;\\
(iv) $J''(\gamma_0, y)< 0$, for every non-null admissible variations $y^i(t)\ne \rho(t)\dot{x}^i(t) (t_0 \le t \le t_1)$, vanishing at $t = t_0$ and $t = t_1$ and satisfying with $\gamma_0$ the equations, 
$$ \int\limits_{t_0}^{t_1}(g_{x^i}y^i+g_{\dot{x}^i}\dot{y}^i)dt= 0, \quad i = 1, 2;$$
(v) along $\gamma_0$ the inequality $\sum\limits_{i,j=1}^{2}h_{\dot{x}^i\dot{x}^j}y^iy^j<0$ holds for all $y\ne k\dot{\gamma_0}(t)$.\\
 Then $\gamma_0$ is a proper strong maximum of the isoperimetric problem.
\end{theorem}

\textbf{Conjugate point along a curve:} A point $\gamma_0(t)$, $t_0 <t \le t_1$, is said to be a \textit{conjugate point} to $\gamma_0(t_0)$ on $\gamma_0$, if there exist a solution $y^i(t)$ of the Jacobi equations
\begin{equation*}
\omega_{y^i}-\frac{d\omega_{\dot{y}^i}}{dt}=0, \quad i=1,2,
\end{equation*}
such that $\int\limits_{t_0}^{t_1}\left( (g)_{x^i}y^i + (g)_{\dot{x}^i}\dot{y}^i\right) dt = 0$.\\

\textbf{Characterizations of conjugate point in terms of Jacobi equations:}
The {\it Jacobi equation} along an isoperimetric extremal admissible curve $\gamma_0$ turns out to be:
\begin{equation}\label{2.1.71}
\Psi(y)+\mu U =0,
\end{equation}
where, 
\begin{equation}\label{2.1.73}
\Psi (y)=h_{2}y-\frac{d}{dt}(h_{1}y'),
\end{equation}
and,
\begin{equation}\label{2.1.70}
\begin{split}
h_{1}:=\frac{h_{\dot{x}^1\dot{x}^{1}}}{(\dot{x}^{2})^{2}},\;\;
h_{2}:=\frac{1}{(\dot{x}^{2})^{2}}\left( h_{x^{1}x^{1}}-(\ddot{x}^{2})^{2}h_{1}-\frac{dK}{dt}\right),\\
 K :=h_{x^{1}\dot{x}^{1}}-\dot{x}^{2} \ddot{x}^{2}h_{1}, \;\;
  U :=g_{x^{1}\dot{x}^{2}}- g_{\dot{x}^{1}x^{2}}-g_{\dot{x}^{1}\dot{x}^{1}}\frac{d}{dt}\left( \frac{\dot{x}^{1}}{\dot{x}^{2}}\right).\\
  \end{split}
\end{equation}
Moreover, the solution $y(t)$ also satisfies, 
\begin{equation}\label{2.1.65}
    \int\limits_{a}^{b}Uydt=0.
\end{equation}

\vspace{0.1in}
Thus, equivalently, if $\gamma_0(c)$ is a conjugate point to the point $\gamma_0(a)$ along the curve $\gamma_0$, then the solution $y(t)$ of the above Jacobi equation with the integral condition (\ref{2.1.65})  must also satisfy the following two conditons:\\
\begin{equation}\label{2.1.66}
y(a) = y(c) = 0,   \quad y(t)\ne 0 \quad \textnormal{on} \quad a < t < c.
\end{equation}

\vspace{0.2in}
Let $y(t)=c_1\theta_1(t)+c_2\theta_2(t)+\mu\theta_3(t)$ be a solution of \eqref{2.1.73}. Then from \eqref{2.1.65} and \eqref{2.1.66} we have
\begin{equation}
    y(a)= c_1\theta_1(a)+c_2\theta_2(a)+\mu\theta_3(a)=0,
\end{equation}
\begin{equation}
    y(c)= c_1\theta_1(c)+c_2\theta_2(c)+\mu\theta_3(c)=0,
\end{equation}
and
\begin{equation}
  \int\limits_{a}^{c}Uydt=c_1\int\limits_{a}^{c}U\theta_1dt+c_2\int\limits_{a}^{c}U\theta_2dt+\mu\int\limits_{a}^{c}U\theta_3dt=0.
\end{equation}
Therefore, if for all admissible curve $y$, $J''\ne 0$, it is necessary that
 \begin{equation}\label{2.1.03}
 D(a,c) =
 \begin{vmatrix}
     \theta_1(a) & \theta_2(a) & \theta_{3}(a)\\ 
     \theta_1(c) & \theta_{2}(c) & \theta_{3}(c)\\
     \int\limits_{a}^{c}U \theta_1dt & \int\limits_{a}^{c}U \theta_2dt & \int\limits_{a}^{c}U\theta_3dt
\end{vmatrix}\ne 0 .   
 \end{equation}
 \textbf{The function $h_1$:}\\
From the Bolza book \cite{OB}, page 119, $h(x^1,x^2,y^1,y^2)$ defined by \eqref{2.03} is positively homogeneous of degree one with respect to $y^1,y^2$. Hence, by the Euler's theorem for homogeneous polynomials of two variables we get,
 \begin{equation}\label{2.3.1}
     y^1h_{y^1}+y^2h_{y^2}=h.
 \end{equation}
 Differentiating \eqref{2.3.1} with respect to $x^1$ and $x^2$ we obtain,
 \begin{equation}
     h_{x^1}=y^1h_{y^1x^1}+y^2h_{y^2x^1}, \quad h_{x^2}=y^1h_{y^1x^2}+y^2h_{y^2x^2}.
 \end{equation}
 And differentiating \eqref{2.3.1} with respect to $y^1$ and $y^2$ we obtain,
 \begin{equation}
    y^1h_{y^1y^1}+y^2h_{y^2y^1}=0, \quad y^1h_{y^1y^2}+y'h_{y^2y^2}=0.
 \end{equation}
 Hence, if $y^1$ and $y^2$ are not both zero,
 \begin{equation}
     h_{y^1y^1}: h_{y^1y^2}:h_{y^2y^2}=(y^2)^2:-y^1y^2:(y^1)^2.
 \end{equation}
 Therefore, there exist a function $h_1$ of $x^1,x^2,y^1,y^2$ such that,
 \begin{equation}\label{eq2.10}
  h_{y^1y^1}=(y^2)^2h_1, \quad h_{y^1y^2}=- y^1y^2h_1, \quad  h_{y^2y^2}=(y^1)^2h_1.
 \end{equation}
 The function $h_1$  thus defined is $C^1$ in the domain of definition of $h$. See pages 120-121 of \cite{OB} for more details.
 \begin{proposition}\label{prop2.4}
 There is no point which is conjugate to $\gamma_0(a), t\in[a,c]$ along the curve $\gamma_0$, if $h_1<0 (>0)$, along $\gamma_0$.
 \end{proposition}
  \begin{proof}
   If along $\gamma_0$, $h_1<0$ and $a<c$, then from equation $(44)$, [$(II'), (III')$ of page $242$] and (75) of  \cite{OB}, it follows that 
  for $a\le t \le c$,  $D(a,c) \neq 0$.
  Therefore, from \eqref{2.1.03} it follows that there are no conjugate points along $\gamma_{0}$.\\
  Similarly, we can show that if $h_1>0$, $\gamma_o(a)$   does not have a conjugate point along $\gamma_0$
 \end{proof}
\begin{theorem}[\cite{LZ1}, \cite{LZ2}]
The inequality $J''(\gamma_0, y) <0$ (or, $>0$) holds for all $y \ne 0$ if and only if there is no point $\gamma_0(c)$ conjugate to $\gamma_0(a)$ along  $\gamma_0$.
\end{theorem}
 \begin{remark}
 \textnormal{Thus we obtain another criteria for detecting conjugate points using calculus of variation techniques by Bolza. It should be noted that so far this criteria is not explored  while solving isoperimetric problems. Thus our approach here differs from others.}
 \end{remark}

 \section{The isoperimetric problem in Randers Poincar\'e disc with respect to different volume forms}
  In this section, we prove Theorem \ref{thm2.1} . We begin by proving the following proposition:\\
  
 \textbf{Proof of Proposition \ref{prop2.001}}
 Since $\beta$ is an exact one form and has constant norm with respect to $\alpha$, we can write
\begin{equation}\label{eq2.1}
    \beta= df= \frac{\partial f}{\partial x^1}dx^1+\frac{\partial f}{\partial x^2}dx^2 \quad \textnormal{and} \quad  \|\beta \|^2_{\alpha}= c^2,
\end{equation}
for some constant $c$. Now we calculate $ \|\beta \|^2_{\alpha}$ as follows:
\begin{equation}\label{eq2.3}
 \|\beta \|^2_{\alpha}=a^{ij}\frac{\partial f}{\partial x^i}\frac{\partial f}{\partial x^j}= \frac{(1-(x^1)^2-(x^2)^2)^2}{4}\left(\frac{\partial f}{\partial x^1}\right)^2+\frac{(1-(x^1)^2-(x^2)^2)^2}{4}\left(\frac{\partial f}{\partial x^2}\right)^2.   
\end{equation} 
Let us assume
\begin{equation}\label{eq2.4}
    r:= \sqrt{(x^1)^2+(x^2)^2} \quad \textnormal{and} \quad  f(x^1,x^2):= \phi(r).
\end{equation} 
Differentiating the above equation with respect to $x^1$ and $x^2$ we get the following,
\begin{equation}\label{eq2.5}
 \frac{\partial r}{\partial x^1} =\frac{x^1}{\sqrt{(x^1)^2+(x^2)^2}}, \quad \frac{\partial r}{\partial x^2} = \frac{x^2}{\sqrt{(x^1)^2+(x^2)^2}} 
\end{equation}
and
\begin{equation}\label{eq2.7}
    \frac{\partial f}{\partial x^1}= \frac{\partial \phi}{\partial r}\frac{\partial r}{\partial x^1}= \frac{x^1}{\sqrt{(x^1)^2+(x^2)^2}}\frac{\partial \phi}{\partial r}, \quad \textnormal{and} \quad \frac{\partial f}{\partial x^2}= \frac{\partial \phi}{\partial r}\frac{\partial r}{\partial x^2}= \frac{x^2}{\sqrt{(x^1)^2+(x^2)^2}}\frac{\partial \phi}{\partial r}.
\end{equation}
Substituting these values in \eqref{eq2.3} we obtain,
\begin{equation}
    (1-r^2)^2\left(\frac{\partial \phi}{\partial r}\right) ^2=c^2.
\end{equation}
Solving the above equation we get,
\begin{equation}
    \phi(r)= \frac{c}{2}\log\frac{1+r}{1-r}.
\end{equation}
Substituting the value of $r$ in the above equation we obtain,
\begin{equation}
    f(x^1,x^2)= \frac{c}{2}\log \frac{1+\sqrt{(x^1)^2+(x^2)^2}}{1-\sqrt{(x^1)^2+(x^2)^2}}.
\end{equation}
Now from $\beta=df$, we obtain $\beta$ as given in the proposition.\\

To prove Proposition \ref{prop2.0011} we need the following  Yasuda-Shimada theorem:\\

\textbf{Yasuda-Shimada theorem} \cite{BCZ} A Randers space with  metric $F=\alpha+\beta$, where $\alpha=\sqrt{a_{ij}y^iy^j}$ and $\beta=b_iy^i$ has constant negative flag curvature $\frac{-\lambda^2}{4}$  if and only if the underlying Riemannian metric $\alpha$ has a constant negative sectional curvature $\frac{-\lambda^2}{4}$, the one-form $\beta$ is exact  and satisfies the system of partial differential equations
\begin{equation}\label{2.1}
    b_{i.x^j}-b_k\gamma^k_{ij}=\lambda(a_{ij}-b_ib_j), \qquad i,j=1,2.
\end{equation}
where, $\gamma^k_{ij}$ are the christoffel symbols of the Riemannian metric $\alpha$.\\

\textbf{Proof of Proposition \ref{prop2.0011}}
The Randers Poincar\'{e} disc $(\mathbb{D}, F=\alpha+\beta)$ has constant curvature if the one-form $\beta$  satisfy \eqref{2.1}.\\
From Proposition \ref{prop2.001} we have 
\begin{equation*}
b_i=\frac{\delta_{ij}x^j}{(1-(x^1)^2-(x^2)^2)^2\sqrt{((x^1)^2+(x^2)^2)^2}}, \quad
a_{ij}=\frac{4\delta_{ij}}{(1-(x^1)^2+(x^2)^2)^2}.    
\end{equation*}
Also for the hyperbolic metric $\alpha$, the Chrystoffel symbols are given by
\begin{equation*}
\left(\gamma^1_{ij}\right)=\frac{2}{1-(x^1)^2-(x^2)^2}\begin{pmatrix}
x^1 & x^2 \\ x^2 & -x^1 \end{pmatrix} \qquad
\textnormal{and} \qquad \left(\gamma^2_{ij}\right)=\frac{2}{1-(x^1)^2-(x^2)^2}\begin{pmatrix}
-x^2 & x^1 \\ x^1 & x^2 \end{pmatrix}.  
\end{equation*}
\\

Thenefore, it can be easily seen that one-form $\beta$ does not satisfy \eqref{2.1}. Hence, we have the proof of the proposition.\\

\begin{proposition}\label{prop2.01}
      Let $(\mathbb{D},F = \alpha+\beta)$ be the Randers  Poincar\'e disc with $\beta=bdf$ given in Proposition \ref{prop2.001}.  Let $\gamma(t) = (x^{1}(t),x^{2}(t))$ be a simple closed curve in $\mathbb{D}$, with $t\in[t_{0},t_{1}]$ and $\gamma(t_{0})=\gamma(t_{1})$. Then\\
      (i) The  Randers length $L$ of the curve  $\gamma$ is given by,
      \begin{equation}\label{2.1.05}
       L(\gamma)=\int\limits_{t_{0}}^{t_{1}}\frac{2\sqrt{(dx^1)^2+(dx^2)^2}}{1-(x^1)^2-(x^2)^2} dt.
      \end{equation}
      (ii) The Busemann-Hausdorff area $A_{BH}$ enclosed by a simple closed curve $\gamma(t)$ is given by,
       \begin{equation}\label{2.1.06}
      A_{BH} =2(1-b^2)^{\frac{3}{2}}\int\limits_{t_{0}}^{t_{1}}\frac{x^{1}\dot{x}^{2}-x^{2}\dot{x}^{1}}{1-(x^1)^2-(x^2)^2}dt.
      \end{equation}
      (iii) The Holmes-Thompson area $A_{HT}$ enclosed by a simple closed curve $\gamma(t)$ is given by,
             \begin{equation}\label{2.1.006}
            A_{HT} =2\int\limits_{t_{0}}^{t_{1}}\frac{x^{1}\dot{x}^{2}-x^{2}\dot{x}^{1}}{1-(x^1)^2-(x^2)^2}dt.
            \end{equation}
(iv) The maximum area $A_{max}$ with respect to the volume form $dV_{max}$ enclosed by a simple closed curve $\gamma(t)$ is given by,
             \begin{equation}\label{2.1.0006}
            A_{max} =2(1+b)^3\int\limits_{t_{0}}^{t_{1}}\frac{x^{1}\dot{x}^{2}-x^{2}\dot{x}^{1}}{1-(x^1)^2-(x^2)^2}dt.
            \end{equation}
(v) The minimum area $A_{min}$ with respect to the volume form $dV_{min}$ enclosed by a simple closed curve $\gamma(t)$ is given by,
             \begin{equation}\label{2.1.1006}
            A_{min} =2(1-b)^3\int\limits_{t_{0}}^{t_{1}}\frac{x^{1}\dot{x}^{2}-x^{2}\dot{x}^{1}}{1-(x^1)^2-(x^2)^2}dt.
            \end{equation}
      \end{proposition}
      \begin{proof}
 $(i)$ From \eqref{eqn2.1} the length of the curve $\gamma$ in the Randers Poincar\'e disc is given by 
 \begin{equation}
 \linebreak
        L(\gamma)=\int\limits_{t_{0}}^{t_{1}}\left( \alpha +\beta\right)  dt \\
        =\int\limits_{t_{0}}^{t_{1}}\frac{2\sqrt{(\dot{x}^1)^2+(\dot{x}^2)^2}}{1-(x^1)^2-(x^2)^2}dt+\int\limits_{t_{0}}^{t_{1}}\dfrac{d}{dt}\left( \log\frac{1+\sqrt{(x^1)^2+(x^2)^2}}{1-\sqrt{(x^1)^2+(x^2)^2}}\right) dt.
       \end{equation}
       As $\gamma(t)$ is a closed curve, $\displaystyle{\int\limits_{t_{0}}^{t_{1}}\dfrac{d}{dt}\left( \log\frac{1+\sqrt{(x^1)^2+(x^1)^2}}{1-\sqrt{(x^1)^2+(x^2)^2}}\right) dt=0}$. Hence $(i)$ is proved.\\
       
        $(ii)$ From Lemma 2.1 the Busemann-Hausdorff volume form of a Randers metric $\alpha+\beta$ is given by $\displaystyle{\frac{4(1-b^2)^{\frac{n+1}{2}}}{(1-(x^1)^2-(x^2)^2)^2}dx}$. Since in our case $n=2$, we have Busemann-Hausdorff volume form of $F$ is $\displaystyle{\frac{4(1-b^2)^{\frac{3}{2}}}{(1-(x^1)^2-(x^2)^2)^2}dx^1dx^2}$. Hence the surface area enclosed by the simple closed curve $\gamma(t)$ is given by
 \begin{equation}
 A_{BH}=\iint\limits_{R}\frac{4(1-b^2)^{\frac{3}{2}}}{(1-(x^1)^2-(x^2)^2)^2}dx^1dx^2,
 \end{equation}
 where $R$ is the region enclosed by the curve $\gamma(t)$.
 Using Green's Theorem, 
 \begin{equation*}
     \int_{\partial R}Pdx^1+Qdx^2=\iint_R\left( \frac{\partial Q}{\partial x^1}-\frac{\partial P}{\partial x^2}\right)dx^1dx^2
 \end{equation*}
and choosing $\displaystyle{P=-\frac{x^2}{1-(x^1)^2-(x^2)^2}}$ and $\displaystyle{Q=\frac{x^1}{1-(x^1)^2-(x^2)^2}}$ we obtain \eqref{2.1.06}.\\

 The proof of $(iii)$, $(iv)$, $(v)$ are analogous to $(ii)$. 
      \end{proof}\\

\subsection{The Isoperimetric problem with respect to the Busemann Hausdorff volume form}
In this subsection, we first solve the isoperimetric problem for the Busemann-Hausdorff volume form. For this we first show that, along the isoperimetric circles the Lagrange multiplier $\lambda<0$. \\
Let us consider the functional given by \eqref{eq2.01}, where $L$ and $A_{BH}$ are given by \eqref{2.1.05} and \eqref{2.1.06}, respectively. We have, by (\ref{2.03})
 \begin{equation*}
h= f+\lambda g, \;\;{\mbox{where}}
\end{equation*}
\begin{equation}\label{2.1.68}
f=2(1-b^2)^{\frac{3}{2}}\frac{x^{1}\dot{x}^{2}-x^{2}\dot{x}^{1}}{1-(x^1)^2-(x^2)^2}, \quad g=\frac{2\sqrt{(dx^1)^2+(dx^2)^2}}{1-(x^1)^2-(x^2)^2}.
\end{equation}
Now we parametrize  the curve $\gamma$ in polar coordinate system as follows:
 \begin{equation}\label{2.08}
\gamma(t) = (x^{1}(t),x^{2}(t)) = (r(t)\cos t, r(t)\sin t).
\end{equation}
Differentiate $\gamma(t)$ with respect to $t$ we get,
\begin{equation}\label{2.09}
\dot{\gamma}(t) = (\dot{x}^{1}(t),\dot{x}^{2}(t)) =(\dot{r}(t)\cos t-r(t)\sin t, \dot{r}(t)\sin t+r(t)\cos t).
\end{equation}
Hence, along $\gamma(t)$  we get,
 \begin{equation}\label{2.10}
x^{1}\dot{x}^2-x^{2}\dot{x}^1 = |x|^2 = r^2,~~ |\dot{x}|^{2} = r^{2}+\dot{r}^{2}.
\end{equation}
 Using \eqref{2.10} in \eqref{2.1.68} we get,
\begin{equation}\label{2.1.67}
h(r,\dot{r}, t) = \frac{2(1-b^2)^\frac{3}{2}r^2}{1-r^2}+ \lambda \frac{\sqrt{r^2+\dot{r}^2}}{1-r^2} .
\end{equation}
The Euler-Lagarange equation of $J$ is,
\begin{equation}\label{2.12}
\frac{\partial h}{\partial r}-\frac{d}{dt}\frac{\partial h}{\partial \dot{r}}=0.
\end{equation}
From the Euler- Lagrange equation given in \eqref{2.12} we have,
\begin{equation}
\begin{split}
\frac{4(1-b^2)^\frac{3}{2}r}{(1-r^2)^2}+ \lambda \frac{r+r^3+2r\dot{r}}{(1-r^2)^2}   -\lambda\frac{d}{dt}\left( \frac{2\dot{r}}{\sqrt{r^2+\dot{r}^2}(1-r^2)}\right)   =0.
\end{split}
\end{equation}
\begin{proposition}
If the circles
$\gamma_0(t)=(a \cos t, a \sin t)$ centered at the origin, where $a \in (0, 1)$ and $t \in [0, 2\pi]$  satisfies the Euler-Lagrange equation, then $\lambda=-\frac{2a(1-b^2)^{\frac{3}{2}}}{1+a^2}$.
\end{proposition}
\begin{proof}
Along $\gamma_0(t)=(a\cos t, a\sin t)$, $r$ is constant. Therefore, the above equation reduces to,
      \begin{equation}\label{2.1.0}
      \lambda(1+a^2) =-2a(1-b^2)^{\frac{3}{2}}.
      \end{equation}
Since, $b<1$, we obtain
 \begin{equation}\label{eq2.22}
 \lambda=-\frac{2a(1-b^2)^{\frac{3}{2}}}{1+a^2}<0 .
 \end{equation}
\end{proof}

\begin{proposition}
The circles
\begin{equation}\label{2.14}
\gamma_0=(a\cos t, a\sin t), \qquad a<1
\end{equation} are normal with respect to the Randers Poincar\'{e} metric.
\end{proposition}
\begin{proof}

Differentiating $g$ with respect to $x^1$, $x^2$, $\dot{x}^{1}$ and $\dot{x}^{2}$ respectively, we get,
\begin{equation}\label{2.1.19}
\begin{split}
g_{x^{1}}=\frac{4x^1\sqrt{(\dot{x}^{1})^2+(\dot{x}^2)^2}}{(1-(x^1)^2-(x^2)^2)^2}, \quad g_{x^{2}}=\frac{4x^2\sqrt{(\dot{x}^{1})^2+(\dot{x}^2)^2}}{(1-(x^1)^2-(x^2)^2)^2}, \\ g_{\dot{x}^{1}}=\frac{2\dot{x}^{1}}{(1-(x^1)^2-(x^2)^2)\sqrt{(\dot{x}^{1})^2+(\dot{x}^{2})^2}}, \quad  g_{\dot{x}^{2}}=\frac{2\dot{x}^{2}}{(1-(x^1)^2-(x^2)^2)\sqrt{(\dot{x}^{1})^2+(\dot{x}^{2})^2}}.
\end{split}
\end{equation}
Along the circle $\gamma_0$ we have,
\begin{equation}\label{2.1.21}
\begin{split}
g_{x^{1}}= \frac{4a^2\cos t}{(1-a^2)^2}, \quad g_{x^{2}}= \frac{4a^2\sin t}{(1-a^2)^2},\\
g_{\dot{x}^{1}}=-\frac{2\sin t}{1-a^2},\quad g_{\dot{x}^{2}}=\frac{2\cos t}{1-a^2}.
\end{split}
\end{equation}
Since,
\begin{equation*}\label{2.1.22}
P_{j}=g_{x^{j}}-\frac{d}{dt}g_{\dot{x}^{j}},\quad \textnormal{for} ~~j=1,2,
\end{equation*}
along $\gamma_0$ we have,
\begin{equation*}
P_{1}=\frac{2(1+a^2)}{(1-a^2)^2}\cos t, \quad P_{2}=\frac{2(1+a^2)}{(1-a^2)^2}\sin t.
\end{equation*}
Therefore, it follows that, 
\begin{equation}\label{2.1.23}
(P_{1},P_{2})\neq 0.
\end{equation}
Hence the circles $\gamma _{0}$ are normal by \eqref{2.1.64}.
\end{proof}

\begin{proposition}\label{prop2.3}
Let $\gamma_{0}$ be the circle centered at the origin in ${\mathbb{D}}$. For each $(x, \dot{x})$ in a neighbourhood of $\gamma_{0}$,  
if  $\lambda <0$, then the Weierstrass function $E$ of the integral $J$ satisfies 
\begin{equation}\label{2.1.26}
E(x^{1},x^{2},\dot{x}^{1},\dot{x}^{2},u^{1},u^{2})\leq 0.
\end{equation}
Moreover, the equality holds if and only if $(u^{1},u^{2})=k(\dot{x}^{1},\dot{x}^{2})$, where $k$ is positive constant.
\end{proposition}
\begin{proof}
In this case we obtain,
\begin{equation}\label{2.1.28}
\frac{\partial h}{\partial \dot{x}^{1}}=-\frac{2(1-b^2)^{\frac{3}{2}}x^2}{1-(x^1)^2-(x^2)^2}+\lambda  \frac{2\dot{x}^{1}}{(1-(x^1)^2-(x^2)^2)\sqrt{(\dot{x}^{1})^{2}+(\dot{x}^{2})^{2}}} 
\end{equation}
and
\begin{equation}\label{2.1.29}
\frac{\partial h}{\partial \dot{x}^{2}}=\frac{2(1-b^2)^{\frac{3}{2}}x^1}{1-(x^1)^2-(x^2)^2}+\lambda  \frac{2\dot{x}^{2}}{(1-(x^1)^2-(x^2)^2)\sqrt{(\dot{x}^{1})^{2}+(\dot{x}^{2})^{2}}}.
\end{equation}
Using \eqref{2.1.28} and \eqref{2.1.29} in \eqref{2.25} and simplifying we get,
\begin{equation}\label{2.1.33}
E(x,\dot{x},u)=2\lambda\left( \sqrt{(u^{1})^{2}+(u^{2})^{2}}-\frac{\dot{x}^{1}u^{1}+\dot{x}^{2}u^{2}}{\sqrt{(\dot{x}^{1})^{2}+(\dot{x}^{2})^{2}}}\right) =\frac{2\lambda}{\| \dot{x}\|}\left( \| u\| \| \dot{x}\| -(\dot{x}^{1}u^{1}+\dot{x}^{2}u^{2})\right).
\end{equation}
Now using  the Cauchy-Schwarz  inequality,
\begin{equation}\label{2.1.34}
(\| u\| \| x\| -(\dot{x}^{1}u^{1}+\dot{x}^{2}u^{2}))\ge 0,
\end{equation}
and the fact that, $\lambda \le 0$, we obtain
\begin{equation}\label{2.1.35}
E(x,\dot{x},u)\le 0.
\end{equation}
It is evident that the equality holds if and only if $u=k\dot{x}$, for any constant $k$. Hence, we conclude the proof. 
\end{proof}


 \begin{proposition}\label{prop2.5}
 Consider $(\mathbb{D},F)$ be the Randers Poincar\'e disc. Then along $\gamma_0$, $h_1=\frac{2\lambda}{a(1-a^2)}<0$.
 \end{proposition}
 \begin{proof}
 In view of \eqref{2.1.68} we have,
 \begin{equation}
     h_{y^1y^1}=\frac{2\lambda \cos^2 t}{a(1-a^2)}, ~ h_{y^1y^2}=\frac{2\lambda \sin t\cos t}{a(1-a^2)}, ~ h_{y^2y^2}=\frac{2\lambda \sin^2 t}{a(1-a^2)}.
 \end{equation}
 Therefore, comparing with \eqref{eq2.10}, we have that $h_1$ along $\gamma_0$ is $\frac{2\lambda}{a(1-a^2)}$ which is less than zero.
Hence, from Proposition \ref{prop2.4} there are no conjugate points along $\gamma_0$.
 \end{proof}

\begin{proposition}\label{prop2.6}
Along the isoperimetric extremal circle $\gamma_{0}$, the following inequality holds:
\begin{equation}\label{2.1.61}
\sum\limits_{i,j =1}^{2}h_{\dot{x}^{i}\dot{x}^{j}}y^{i}y^{j}\leq 0,
\end{equation}
and the equality holds if and only if $(y^{1},y^{2})=k(\dot{x}^1,\dot{x}^2)$.
\end{proposition}
\begin{proof}
Differentiating \eqref{2.1.28} with respect to $\dot{x}^1$ and \eqref{2.1.29} with respect to $\dot{x}^2$ yields,
\begin{equation}
h_{\dot{x}^1\dot{x}^1}= \frac{\lambda (\dot{x}^2)^2 }{((\dot{x}^1)^2+(\dot{x}^2)^2)^{\frac{3}{2}}}, ~ h_{\dot{x}^1\dot{x}^2}=-\frac{ \dot{x}^1\dot{x}^2 }{((\dot{x}^1)^2+(\dot{x}^2)^2)^{\frac{3}{2}}} , ~h_{\dot{x}^2\dot{x}^2}= \frac{\lambda (\dot{x}^1)^2 }{((\dot{x}^1)^2+(\dot{x}^2)^2)^{\frac{3}{2}}}.
\end{equation}
Therefore,
 \begin{equation}\label{2.1.63}
\sum\limits_{i,j =1}^{2}h_{\dot{x}^{i}\dot{x}^{j}}y^{i}y^{j} = \frac{\lambda}{((\dot{x}^1)^2+(\dot{x}^2)^2)^{\frac{3}{2}}}(\dot{x}^2y^1-\dot{x}^1y^2)^2 = \frac{\lambda}{a}(y^{1}\cos t+y^{2}\sin t)^{2}.
\end{equation}
As $\lambda$ is negative and $a>0$, the inequality \eqref{2.1.61} follows. Clearly, the equality holds if and only if $(y^1,y^2)=k(\dot{x}^1,\dot{x}^2)$.
\end{proof}

\subsection*{\textbf{Proof of Theorem \ref{thm2.1}}}
In Propositions \ref{prop2.01}-\ref{prop2.6}, we have shown that $\gamma_0$ satisfies all the conditions of Theorem \ref{thm2.2}. Hence, $\gamma_0$ is a solution of the isoperimetric problem with respect to the Busemann-Hausdorff measure. Similarly, it also follows that $\gamma_0$ is a solution of the isoperimetric problem with respect to the Holmes-Thompson, the maximum, the minimum volume forms as these volume forms are just the scalar multiples of the Busemann-Hausdorff volume form.

 \end{document}